\documentclass[a4paper, 11pt, parskip=half]{scrartcl}

\usepackage{custom-template, shortcuts,graphicx}

\begin{document}

\title{\vspace{-40pt}Lafforgue pseudocharacters and parities of limits of Galois representations}
\author{Tobias Berger and Ariel Weiss}
\date{}
\maketitle

\begin{abstract}
	\vspace{-15pt}
Let $F$ be a CM field with totally real subfield $F^+$ and let $\pi$ be a $C$-algebraic cuspidal automorphic automorphic representation of $\U(a,b)(\A_{F^+})$ whose archimedean components lie in the (non-degenerate limit of) discrete series. We attach to $\pi$ a Galois representation $R_\pi:\GalF\to\CU(a,b)(\Qlb)$ such that, for any complex conjugation element $c$, $R_\pi(c)$ is as predicted by the Buzzard--Gee conjecture \cite{BuG}. As a corollary, we deduce that the Galois representations attached to certain irregular, $C$-algebraic (essentially) conjugate self-dual cuspidal automorphic representations of $\GL_n(\A_F)$ are \emph{odd} in the sense of Bella\"iche--Chenevier \cite{BCh}.
\end{abstract}
\section{Introduction}

If $\rho :\Ga \Q\to\GL_2(\Qlb)$ is the $\l$-adic Galois representation attached to a classical modular eigenform, then $\rho$ is \emph{odd}: for any choice of complex conjugation $c\in \Ga\Q$, we have $\det(\rho(c)) = -1$. For Galois representations attached to automorphic representations of general reductive groups, this notion of oddness has been generalised by Buzzard--Gee \cite{BuG}, where it is interpreted as local-global compatibility at the archimedean place of $\Q$.

In this paper, we study the image of complex conjugation for Galois representations attached to certain irregular automorphic representations of unitary groups. Our main result is the following theorem:

\begin{theorem}\label{main-thm-intro}
	Let  $F$ be a CM field with totally real subfield $F^+$. Let $\pi$ be a $C$-algebraic cuspidal automorphic representation of $\U(a,b)(\A_{F^+})$ such that, for each archimedean place $v$ of $F^+$, $\pi_v$ is a (non-degenerate limit of) discrete series representation. Let $\l$ be a prime at which $\pi$ is unramified. Then there exists a  Galois representation
	$$R_{\pi}:\GalF\to \CU(a,b)(\Qlb)$$
	attached to $\pi$ that satisfies local-global compatibility at unramified primes and such that, for any complex conjugation element $c\in\GalF$, $R_{\pi}(c)$ is as predicted by the Buzzard--Gee conjecture \cite[Conj. 5.3.4]{BuG}.
\end{theorem}

If $\pi$ is an automorphic representation of $\U(a,b)(\A_{F^+})$, such that $\pi_v$ is a non-degenerate limit of discrete series representation and $\l$ is a prime at which $\pi$ is unramified, then Goldring--Koskivirta \cite[Theorem 10.5.3]{GoldringKoskivirta} have recently attached an $\l$-adic Galois representation
$$\rho_{\pi}: \Ga F\to \GL_n(\Qlb)$$
to $\pi$. However, since $\pi$ is an automorphic representation over $F^+$, its associated Galois representation should be a representation of $\GalF$. Our contribution is to extend $\rho_{\pi}$ to a representation of $\GalF$ with the correct sign at infinity.

\subsection{The sign of a conjugate self-dual representation}
 
The representation $\rho_{\pi}$ constructed by Goldring--Koskivirta is \emph{essentially conjugate self-dual}: there exists a character $\chi$ of $\Ga F$, depending on $\pi$, with $\chi \simeq\chi^c$, such that $\rho_{\pi}^c \simeq \rho_{\pi}\dual\tensor\chi$. If $\rho_{\pi}$ is \emph{absolutely irreducible}, then Bella\"iche--Chenevier \cite{BCh} have introduced the following notion of the sign of $\rho_{\pi}$: since $\rho_{\pi}$ is irreducible, by Schur's Lemma, there is a matrix $A\in\GL_n(\Qlb)$, unique up to scalar multiplication, such that $\rho_{\pi}^c=A\rho_{\pi}\dual A\ii\chi$. Applying this relation twice, we see that $AA\iit$ commutes with $\rho_{\pi}$ and, hence, by Schur's Lemma again, that $A^t = \lambda A$, where $\lambda = \pm1$. We call $\lambda$ the Bella\"iche--Chenevier sign of $\rho_{\pi}$ and call $\rho_{\pi}$ \emph{odd} if $\lambda = 1$. In  \Cref{polarisable-sign-bijection}, we will see that $\rho_{\pi}$ being odd is equivalent to $\rho_{\pi}$ lifting to a representation $\GalF\to\CU(a,b)(\Qlb)$ that satisfies the Buzzard--Gee conjecture at infinity. Applying \Cref{main-thm-intro}, we deduce the following theorem, which generalises the main result of \cite{BCh}.

\begin{theorem}\label{BC-signs-intro}
Let $\pi$ be a $C$-algebraic essentially conjugate self-dual cuspidal automorphic representation of $\GL_n(\A_F)$. Denote by $\chi:\A_F\t \to \C\t$ the character such that $\pi^c\cong\pi\dual\tensor\chi$. Assume that, for each archimedean place $v$ of $F$, $\pi_v$ descends to a $C$-algebraic (non-degenerate limit of) discrete series representation of $\U(a,b)$. Let $\l$ be a prime at which $\pi$ is unramified. Then there exists a Galois representation
	$$\rho_{\pi}: \Ga F\to\GL_n(\Qlb)$$
	attached to $\pi$ such that	$\rho_{\pi}^c \simeq \rho_{\pi}\dual\tensor \rho_{\chi}\epsilon^{1-n}$, where $\epsilon$ is the $\l$-adic cyclotomic character. Moreover, there exists a totally odd polarisation of $\rho_{\pi}$. In particular, if $r$ is an irreducible subrepresentation of $\rho_{\pi}$ that satisfies $r^c\simeq r\dual\tensor \rho_{\chi}\epsilon^{1-n}$ and appears with multiplicity one in the decomposition of $\rho_{\pi}$ into irreducible subrepresentations, then $r$ is odd.
\end{theorem}

When $\pi$ is regular, \Cref{BC-signs-intro} is the main result of \cite{BCh}; in that case, the uniqueness of $r$ in the decomposition of $\rho_{\pi}$ is automatic by the regularity of $\rho_{\pi}$. For an application of the oddness of these Galois representations see \cite{Berger2018}, in particular Remark 2.7.

\subsection{Our method}

When $\pi$ is irregular, $\rho_{\pi}$ is constructed, via its corresponding pseudocharacter, as a limit of Galois representations attached to regular automorphic representations. Although the Galois representations attached to regular automorphic representations are known to be odd \cite{BCh}, it is not clear that this property should be preserved after taking a limit: oddness is not encoded in the trace of $\rho_{\pi}$; moreover, we do not know that $\rho_{\pi}$ is irreducible and, hence, a priori, it need not have a sign at all, nor any lift to a representation valued in $\CU(a,b)(\Qlb)$.

Our solution to these problems is to work with Lafforgue's pseudocharacters \cite{lafforgue2012chtoucas} in place of Taylor's classical pseudocharacters \cite{taylor1991galois}. By the work of Goldring--Koskivirta, for each $n\iN$, the system of Hecke eigenvalues of $\pi$ is congruent modulo $\l^n$ to the system of Hecke eigenvalues of a mod $\l^n$ cohomological eigenform $\pi_n$. By \cite[Theorem 1.2]{BCh} and \Cref{polarisable-sign-bijection}, the Galois representation $\rho_n$ attached to $\pi_n$ lifts to a representation $R_n$ that is valued in $\CU(a,b)$ with the correct sign. Finally, a computation in invariant theory shows that the limit of a sequence of $\CU(a,b)$-valued representations is valued in $\CU(a,b)$ and that the sign is preserved in the limit.

\subsection*{Acknowledgements}
We would like to thank Ga\"etan Chenevier, Wushi Goldring, Sug-Woo Shin, Beno\^it Stroh, and Jack Thorne for helpful discussions related to the topics of this article. The second author was supported by an Emily Erskine Endowment Fund postdoctoral fellowship.

%
%
%
%

\section{Preliminaries}\label{prelims}

\subsection{Notation}
For each prime $\ell$, we fix once and for all an isomorphism $\Qlb \cong \C$. For a number field $L$ and Hecke character $\chi:\A_L\t \to \C\t$ we write $\rho_{\chi}: \Ga L \to \Qlb\t$ for the corresponding Galois character.

\subsection{Unitary groups}

Let $F$ be a CM field with totally real subfield $F^+$. For $x\in F$, let $\overline{x}$ denote the image of $x$ under the non-trivial element of $\Gal(F/F^+)$. Fix an integer $n\iN$ and a matrix $J\in \GL_n(F)$ with $\overline{J}=J^t$. 

\begin{definition}
	The \emph{unitary group} $\U(J)$ is the algebraic group over $F^+$ whose $R$-points are
	$$\U(J)(R)=\set{g\in \GL_n(R\tensor_{F^+}F):gJ\overline{g}^t = J}$$
	for any $F^+$-algebra $R$.
\end{definition}

\begin{definition}
	The \emph{general unitary group} $\GU(J)$ is the algebraic group over $F^+$ whose $R$-points are
	$$\GU(J)(R)=\set{g\in \GL_n(R\tensor_{F^+}F):gJ\overline{g}^t = \lambda J,\; \lambda\in R\t}$$
	for any $F^+$-algebra $R$.
\end{definition}

Typically, we take $J = J_{a,b} = \begin{pmatrix}
I_a&\\&-I_b
\end{pmatrix}$ and write $\U(a,b)$ for the corresponding unitary group. 

For any $F$-algebra $R$, the canonical isomorphism
\begin{align*}
	R\tensor_{F^+}F&\xrightarrow{\sim} R\+R\\
	(r\tensor x)&\mapsto (rx, r\overline{x}).
\end{align*}
allows us to identify $\U(J)_{/F}$ with $\GL_n$ and $\GU(J)_{/F}$ with $\GL_n\times\GL_1$. 


\subsubsection{Root data and the $L$-group}

If $B$ is the upper-triangular Borel of $\U(a,b)_{/\overline{F}}$ and $T$ is the diagonal torus, then the based root datum of $\U(a,b)_{/\overline{F}}\cong \GL_n$ is given by $\Psi(B, T) = (X^*, \Delta^*, X_*, \Delta_*)$ with 

\begin{itemize}
	\item $X^* = \set{\dmatthree{t_1}{\ddots}{t_n}\mapsto t_1^{a_1}\cdots t_n^{a_n} : a_i\iZ}$;
	\item $\Delta^* = \set{E_i - E_{i+1}:i = 1,\ldots, n-1}$, where $E_i$ denotes the character $\dmatthree{t_1}{\ddots}{t_n}\mapsto t_i$;
	\item $X_* = \set{t\mapsto \dmatthree{t^{a_1}}{\ddots}{t^{a_n}}: a_i\iZ}$;
	\item $\Delta_*=\set{E_i' - E_{i+1}'}$, where $E_i'$ denotes the cocharacter $t\mapsto\mathrm{diag}(1, \ldots, 1,t,1,\ldots,1)$ with the non-trivial part in the $i\th$ position.
\end{itemize}

The identification of $\U(a,b)_{/\overline F}$ with $\GL_n$ identifies the dual group $\widehat{\U(a,b)}$ with $\GL_n$, however the action of Galois is different.

\begin{definition}
	The $L$-group of $\U(a,b)$ is
	$$\LU = \LU(a,b) = \GL_n\rtimes \GalF,$$
	where $\GalF$ acts, via its quotient $\Gal(F/F^+)$, by
	$$c\cdot g= \Phi_n g\iit\Phi_n\ii,$$
	where $g\in\GL_n$, $c$ is the non-trivial element of $\Gal(F/F^+)$ and $\Phi_n$ is the matrix whose $ij\th$ entry is $(-1)^{i+1}\delta_{i,n-j+1}$.
\end{definition}

We denote an element of $\LU$ by $(g)\gamma$, where $g\in\GL_n$ and $\gamma\in \GalF$. Note that
$$
(g)\gamma\cdot (I_n)\gamma'= \begin{cases}(I_n) \gamma'\cdot(g) \gamma&\gamma'\in\Ga F\\
(I_n) \gamma'\cdot(\Phi_n g\iit\Phi_n\ii) \gamma&\gamma'\in \GalF\setminus \Ga F.
\end{cases}
$$

\begin{definition}
	The $L$-group of $\GU(a,b)$ is
	$$\LGU = \LGU(a,b) = (\GL_n\times\GL_1)\rtimes \GalF,$$
	where $\GalF$ acts, via its quotient $\Gal(F/F^+)$, by
	$$c\cdot (g, \lambda)= (\Phi_n g\iit\Phi_n\ii,\det(g)\lambda),$$
	where $(g, \lambda)\in\GL_n\times\GL_1$ and $c$ is the non-trivial element of $\Gal(F/F^+)$.
\end{definition}

We note that the $L$-groups do not depend on the signature.

\subsection{Algebraic automorphic representations and the $C$-group}

The Langlands conjectures predict a relationship between \emph{algebraic} automorphic representations and Galois representations. There are two natural notions of what it means to be algebraic, which Buzzard--Gee \cite{BuG} call $L$-algebraic and $C$-algebraic. When $G = \GL_n$, there is a simple method to go from a $C$-algebraic automorphic representation of $G$ to an $L$-algebraic one. However, in general, and in particular when $G = \U(a,b)$, these notions are indeed distinct. In particular, the Galois representations attached to $C$-algebraic automorphic representations of $\U(a,b)$ will not be valued in the $L$-group of $\U(a,b)$ but in its $C$-group, defined in \cite{BuG}. In this subsection, we recall the notions of $L$-algebraic and $C$-algebraic representations and define the $C$-group of $\U(a,b)$.

\subsubsection{Algebraic automorphic representations}

Let $k$ be either $\R$ or $\C$. Let $G$ be a reductive group over a number field $F$, with fixed maximal torus $T$, Borel $B$ and based root datum $\Psi(B, T) = (X^*, \Delta^*, X_*, \Delta_*)$. To an irreducible admissible complex representation of $G(k)$, Langlands associates a $\widehat{G}(\C)$-conjugacy class of admissible homomorphisms
$$r = r_\pi:W_k\to  {}^LG(\C),$$
where 
$$W_k = \begin{cases}
\C\t&k =\C\\
\C\t\sqcup j\C\t &k=\R
\end{cases}$$
is the Weil group of $k$; if $k = \R$, then $j^2 = -1$ and $jzj\ii = \overline z$ for $z\iC\t$.

Fix a maximal torus $\widehat{T}$ in $\widehat{G}_\C$ equipped with an identification $X_*(\widehat{T})=X^*(T)$, and conjugate $r$ so that $r(\C\t)\subset \widehat{T}(\C)$. We find that, for $z\in \C\t$, 
$$r(z) = \lambda(z)\lambda_c(\overline{z}),$$
where $\lambda, \lambda_c\in X_*(\widehat{T})\tensor \C$ and $\lambda-\lambda_c \in X_*(\widehat{T})$.
\begin{definition} \label{CLalgebraic}
	We say that $\pi$ is \emph{$L$-algebraic} if $\lambda, \lambda_c \in X_*(\widehat{T})$. If $\delta$ denotes half the sum of the positive roots, we say that $\pi$ is \emph{$C$-algebraic} if $\lambda, \lambda_c \in\delta+ X_*(\widehat{T})$.

	If $\pi$ is an automorphic representation of $G(\A_F)$, we say that $\pi$ is $L$-algebraic (resp. $C$-algebraic) if $\pi_v$ is $L$-algebraic (resp. $C$-algebraic) for every archimedean place $v$.
\end{definition}

\subsubsection{Twisting elements and the $C$-group}\label{twisting}

If half the sum of the positive roots $\delta$ is itself a root, then the notions of $L$-algebraic and $C$-algebraic coincide. More generally, if $X^*(T)$ contains a twisting element $\theta$ (see \cite[\S 5.2]{BuG}), then we can go between $L$-algebraic and $C$-algebraic representations by twisting by the character $|\cdot|^{\theta-\delta}$.

\begin{proposition}
	Let $G = \U(a,b)$. If $n=a+b$ is odd, then $\delta\in X^*(T)$. If $n$ is even, then $X^*(T)$ does not contain a twisting element. 
\end{proposition}

\begin{proof}
	Recall that an element $\theta \in X^*(T)$ is a twisting element if it is $\GalF$-stable and $\langle\theta,\alpha\dual\rangle = 1\iZ$ for all simple coroots $\alpha\dual$.
	
	Identify $X^*(T)$ with $\Z^{n}$ in the obvious way. Then $\delta = \frac12(n-1, n-3, n-5, \ldots, -n+3, -n+1)\in X^*(T)$ if and only if $n$ is odd.
	
	Suppose that $n$ is even and that $\theta = (a_1, \ldots, a_n)$ is a twisting element. Then, for each each $i = 1, \ldots n-1$, since $\langle\theta, E_{i}'-E'_{i+1}\rangle = 1$, we have $a_i = a_{i+1} + 1$. Hence, $\theta = (a_1, a_1-1, \ldots, a_1-n+1)$. It is clear that no element of this form can by stable under the action of Galois: we have 
	$$c\cdot\theta = (n-1-a_1, n-2-a_1, \ldots, -a_1),$$
	so if $c\cdot \theta = \theta$, then $a_1 = \frac{n-1}2$, which is a root only if $n$ is odd.
\end{proof}

Hence, in general, we cannot go between $L$-algebraic and $C$-algebraic automorphic representations of $\U(a,b)$. To solve this problem, Buzzard--Gee construct an extension $\widetilde{\U(a,b)}$, such that
$$1\to \GG_m\to \widetilde{\U(a,b)}\to \U(a,b)\to 1$$
is exact and $\widetilde{\U(a,b)}$ contains a twisting element. The $C$-group of $G$, $\CG$, is then defined to be the $L$-group of $\widetilde{\U(a,b)}$. In our case, following \cite[Prop 5.3.3]{BuG}, we find that 
$$\CU = \CU(a,b)\cong \widehat{\widetilde{U(a,b)}}\rtimes\GalF,$$
where
$$\widehat{\widetilde{U(a,b)}} \cong \frac{\GL_n\times\GL_1}{\langle(-I_n)^{n-1},  -1\rangle},$$
and $\GalF$ acts via the quotient $\Gal(F/F^+)$: if $c$ is the non-trivial element of $\Gal(F/F^+)$ and $(g,\mu)\in \widetilde{U(a,b)}$, then
$$c\cdot(g, \mu) = (\Phi_n g\iit\Phi_n\ii, \mu).$$
The map $\GG_m\to \widetilde{U(a,b)}$ induces a map $d:\CU(a,b)\to \GG_m$, given by
$$(g,\mu)\times \gamma\mapsto \mu^2.$$
Let $\widetilde{T}$ be the pullback of the torus $T$. Let $\hat\xi\in X_*(\widehat{\widetilde{T}})$ be given by
$$z\mapsto (1,z).$$
Let $\theta = \delta + \frac12\hat\xi\in X_*(\widehat{\widetilde{T}})$; explicitly,
$$\theta: z\mapsto \br{\dmat{z^{\frac{n-1}2}}{z^{\frac{n-3}2}}{\ddots}{z^{\frac{1-n}2}}, z^{\frac12}},$$
which is well-defined by the definition of $\widehat{\widetilde{\U(a,b)}}$. Then $\theta$ is a twisting element.

In the specific case of an automorphic representation $\pi$ of $\U(a,b)$, we now make explicit what it means for $\pi$ to be $C$- or $L$-algebraic. For $z=r e^{i \theta} \in \C$ with $r \in \R_{>0}$ and $a \in \Z$ we write $(z/\overline{z})^{a/2}$ for $e^{ia \theta}$.
\begin{lemma}\label{weakreg}
	Consider $r_{\pi}:W_{\R} \to {}^L\U(a,b)(\C)$ such that 
	$$z \in \C\t \mapsto \dmat{(z/\overline{z})^{a_1}}{(z/\overline{z})^{a_2}}{\ddots}{(z/\overline{z})^{a_n}}$$
	with $a_i \in \frac{1}{2}\Z$.

\begin{enumerate}
\item If $\pi$ is $C$-algebraic (i.e. $a_i \in \Z+\frac{n-1}{2}$), then $r(j)=(A\Phi_n^{-1}) c$ with $A=A^t$.
\item If $\pi$ is $L$-algebraic (i.e. $a_i \in \Z$), then $r(j)=(A\Phi_n^{-1}) c$ with $A=(-1)^{n-1}A^t$.
\item Assume that there exists $i \in \{1, \ldots, n\}$ such that $a_i \neq a_j$ for all $j \neq i$. Then $\pi$ is $C$-algebraic.
\end{enumerate}
\end{lemma}

\begin{proof}
Writing $r(j)=(A \Phi_n^{-1})c$ the semi-direct product relation implies $r(j^2)=(AA^{-t}(-1)^{n-1})\1$.

Since $j^2=-1$ we see that $A=A^t$ if $\pi$ is $C$-algebraic (since $a_i \in \Z+\frac{n-1}{2}$ we have $r(-1)=(-1)^{n-1}$) and $A=(-1)^{n-1}A^t$ if $\pi$ is $L$-algebraic (since $a_i \in \Z$).

For $(3)$ we note that the relationship $jzj\ii=\overline z$ implies
$$Ar(z)=r(z)A.$$
Assuming, without loss of generality, that $a_1 \neq a_2$ this shows that $A$ is of the form $$A=\begin{pmatrix}A_1&0&*\\0&A_2&*\\ *&*&*\end{pmatrix}.$$ In particular, it cannot satisfy $A^t=-A$.  
\end{proof}

\subsection{Galois representations attached to automorphic representations of $\U(a,b)$}


Let $F$ be a CM field, i.e. $F$ is a totally imaginary quadratic extension of a totally real subfield $F^+$, and let $\pi$ be a cuspidal automorphic representation of $\U(a,b)(\A_{F^+})$. In this subsection, we recall results  associating $\Ga F$-representations to $\pi$.

Let $H = \Res_{F/F^+} (\GL_n)$. For an automorphic representation $\pi$, let $\chi_{\pi}$ denote its central character. For every place $v \notin \mathrm{Ram}(\pi)$ and $w$ a place of $F$ above $v$, define the base change of $\pi_v$ to $H$, denoted $\mathrm{BC}(\pi_v)$, and its $w$-part $\mathrm{BC}(\pi_v)_w$ as in \cite[Section 1.3]{HLTT}. Write $\rec_{F_w}$ for the (unramified) local Langlands
correspondence, normalized as in \cite{harris-taylor}.

The following theorem is the work of many people; for a reference see e.g. \cite{HLTT} Corollary 1.3 or \cite{Skinner} (but we state a version over general CM fields for the  regular discrete series case covered by \cite{Shin}, which only requires Labesse’s restricted base change):

\begin{theorem}[Clozel, Harris, Taylor, Labessse, Morel, Shin]\label{Gal-reps}
	Let $\pi$ be a  cuspidal automorphic representation of $\U(a,b)(\A_{F^+})$. Let $S$ be the set of primes of $F$ lying above rational primes at which $F$ and $\pi$ are unramified. Suppose that, for each archimedean place $v$ of $F^+$, $\pi_v$ is a regular discrete series representation.
There exists a compatible system of $\l$-adic Galois representations
		$$\rho_{\pi}:\Ga F\to \GL_n(\Qlb)$$ 
		such that
		$$\rho_{\pi}^c \simeq \rho_{\pi}\dual\tensor\epsilon^{1-n}$$ 
		and such that
		$$(\rho_{\pi}|_{W_{F_w}})^{\mathrm{ss}} \cong \rec_{F_w}(\mathrm{BC}(\pi_v)_w \otimes | \cdot |_w^{\frac{1-n}{2}})$$ for $w \notin S$ and $w \mid v$. These representations are de Rham for primes above $\ell$.
\end{theorem}

For imaginary CM fields, stronger local-global compatibility statements can be proved, see \cite{Shah} Theorem 2.2.
 
Using Theorem \ref{Gal-reps} \cite{GoldringKoskivirta} prove the following result (a similar result is proved by Pilloni-Stroh \cite{Pilloni-Stroh}):
\begin{theorem}[\cite{GoldringKoskivirta} Theorem 10.5.3]\label{GKPS}
	Let $\pi$ be a $C$-algebraic cuspidal automorphic representation of $\U(a,b)(\A_{F^+})$. Let $S$ be the set of primes of $F$ lying above rational primes at which $F$ and $\pi$ are unramified. Suppose that, for each archimedean place $v$ of $F^+$, $\pi_v$ is a non-degenerate limit of discrete series.
 Then, for each prime $\l$ at which $\pi$ is unramified, there exists an $\l$-adic Galois representation
		$$\rho_{\pi}:\Ga F\to \GL_n(\Qlb)$$
		such that
		$$\rho_{\pi}^c \simeq \rho_{\pi}\dual\tensor\epsilon^{1-n}$$
		and such that
		$$(\rho_{\pi}|_{W_{F_w}})^{\mathrm{ss}} \cong \mathrm{rec}_{F_w}(\mathrm{BC}(\pi_v)_w \otimes | \cdot |_w^{\frac{1-n}{2}})$$ for $w \notin S$.
\end{theorem}

\begin{remark}\label{mult2}
	Note that the condition that $\pi$ is $C$-algebraic is often satisfied automatically. Indeed, by \cite[Section 10.5.3]{GoldringKoskivirta}  non-degenerate limits of discrete series correspond to Langlands parameters as in Lemma \ref{weakreg}, with parameters $a_i$ of multiplicity at most two (whereas discrete series have all $a_i, i=1, \ldots n$ distinct). By Lemma \ref{weakreg}, these representations are automatically $C$-algebraic unless each $a_i$ has multiplicity exactly $2$. 
\end{remark}

\subsection{Polarised Galois representations and the Bella\"iche--Chenevier sign}

In the previous section, we recalled the existence of Galois representations
$$\rho:\Ga F\to \GL_n(\Qlb)$$
attached to automorphic representations $\pi$ of $\U(a,b)$. In this subsection, we show how to lift these representations to representations
$$R:\GalF\to\CU(\Qlb)$$
and we relate the image of complex conjugation elements under $R$ to the Bella\"iche--Chenevier sign of $\rho$. 

\subsubsection{Polarised Galois representations}

We begin by recalling the notion of a polarised Galois representation of $\Ga F$.

\begin{definition}
	A \emph{polarised} $\l$-adic Galois representation of $\Ga F$ is a triple $(\rho, \chi, \langle\cdot,\cdot\rangle)$, where
	\begin{itemize}
		\item $\rho:\Ga F\to \GL_n(\Qlb)$ is a Galois representation;
		\item $\chi: \GalF\to \Qlb\t$ is a Galois character;
		\item $\langle\cdot,\cdot\rangle$ is a pairing on $\Qlb^n$
	\end{itemize}
	such that for all $x,y \in \Qlb^n$:
	\begin{itemize}
		\item $\langle x, y \rangle = -\chi(c)\langle y, x\rangle$, where $c$ is a choice of complex conjugation in $\GalF$.
		\item $\langle \rho(g)x, \rho^{c}(g)y \rangle = \chi(g)\langle x, y\rangle$ for all $g\in \Ga F$.
	\end{itemize}
\end{definition}

If $(\rho, \chi, \langle\cdot,\cdot\rangle)$ is a polarised Galois representation, then there is a matrix $A\in\GL_n$, such that, for all $x, y\in \Qlb^n$,
$$\langle x, y\rangle = x^t A\ii y.$$
Since $\langle \rho(g)x, \rho^{c}(g)y \rangle = \chi(g)\langle x, y\rangle$ for all $g\in \Ga F$, we see that
$$\rho^c = A\rho\dual A\ii\chi,$$
so that $\rho$ is conjugate self-dual. Moreover, the condition that $\langle x, y \rangle = -\chi(c)\langle y, x\rangle$, where $c$ is the non-trivial element of $\Gal(F/F^+)$, ensures that $x^tA\ii y =-\chi(c) x^t A\iit y$.
Since $x, y\in\Qlb^n$ were arbitrary, we see that $A = -\chi(c)A^t$. We call $-\chi(c)$ the \emph{sign} of $(\rho, \chi, \langle\cdot,\cdot\rangle)$. If $\rho$ is irreducible, then this sign is exactly the Bella\"iche--Chenevier sign.

If $\rho:\Ga F\to \GL_n(\Qlb)$ is an irreducible, essentially conjugate self-dual Galois representation, then there is a natural way to extend $\rho$ to a polarised Galois representation. Indeed, there is a matrix $A$, unique up to scalar multiplication, such that $\rho^c = A\rho\dual A\ii \chi$ and such that $A = \lambda A^t$ where $\lambda=\pm1$. Since $\chi = \chi^c$, $\chi$ extends to a character of $\GalF$, and we choose this extension so that $\chi(c) = -\lambda$. If we define a pairing $\langle\cdot,\cdot\rangle$ on $\Qlb$ using $A\ii$, then $(\rho, \chi, \langle\cdot,\cdot\rangle)$ is a polarised Galois representation.

More generally, if $\rho$ is semisimple and every irreducible subrepresentation $r$ of $\rho$ such that $r^c\simeq r\tensor \chi$ has sign $\lambda$, then there is still a choice of polarisation for $\rho$. Indeed, we can write
$$\rho = \bigoplus_i r_i \bigoplus_j s_j \+ (s_j^c)\dual\chi,$$
where the $r_i$ are conjugate self-dual with sign $\lambda$. We can define a polarisation on each $r_i$ as before. For each $j$, if $\dim(s_j)=n_j$, then the matrix
$$\begin{pmatrix}
&I_{n_j}\\I_{n_j}&
\end{pmatrix}$$ 
defines an invariant pairing on the conjugate self-dual representation $s_j \+ (s_j^c)\dual\chi$. Taking the direct sum of these polarised Galois representations gives a polarisation of $\rho$ with the correct sign.

\begin{remark}\label{distinct-warning}
	In general, the converse of this construction fails. Given a polarised Galois representation $(\rho, \chi, \langle\cdot,\cdot\rangle)$ with sign $\lambda$, it is \emph{not} true in general that every conjugate self-dual subrepresentation of $\rho$ has Bella\"iche--Chenevier sign $\lambda$. For example, if $r$ is a conjugate self-dual Galois representation with sign $-1$, then we can define two polarisations on $\rho = r\+ r$ with different signs. Indeed, if $r^c = Br B\ii\chi$ with $B= -B^t$, then let
	$$A_1 = \begin{pmatrix}
	B&\\&B
	\end{pmatrix}$$
	and 
	$$A_2 = \begin{pmatrix}
	&-B\\B&
	\end{pmatrix}.$$
	Then the pairing $\langle x, y\rangle  = x^tA_1\ii y$ has sign $-1$, while the pairing $\langle x, y\rangle  = x^tA_2\ii y$ has sign $1$. 
\end{remark}

Nevertheless, given a polarised Galois representation with sign $\lambda$, any subrepresentation $r$ of $\rho$ that is conjugate self-dual and appears with multiplicity $1$ in the decomposition of $\rho$ will have sign $\lambda$. We record this fact in the following lemma.

\begin{lemma}\label{polarisation-determines-sign}
	Let $(\rho, \chi, \langle\cdot,\cdot\rangle)$ be a polarised Galois representation with sign $\lambda$. Suppose that $r$ is an irreducible subrepresentation of $\rho$ that appears with multiplicity $1$ in the decomposition of $\rho$ and such that $r^c\simeq r\dual\tensor \chi$. Then $r$ has sign $\lambda$.
\end{lemma}

\begin{proof}
	Write $\rho = \bigoplus_i r_i \bigoplus_j s_j \+ (s_j^c)\dual\chi$,
	where the $r_i$ are conjugate self-dual. Write $A$ for the matrix such that $\langle x, y\rangle  = x^tA\ii y$. Then
	$$\rho^c = A\rho\dual A\ii\chi.$$
	In particular, $A$ permutes the $r_i$'s and, since $r$ has multiplicity $1$ in the decomposition of $\rho$, there must be a submatrix $A_r$ of the block diagonal of $A$ such that $r^c = A_rr\dual A_r\ii\chi$. In particular, $A_r = \lambda A_r^t$, so $r$ has sign $\lambda$.
\end{proof}

\subsubsection{Galois representations valued in $\CU(a,b)$}

Recall that $\CU = \frac{\GL_n\times\GL_1}{\langle (-I_n)^{n-1}, -1\rangle}\rtimes \GalF$, where the action of Galois is given by
$$c\cdot (g, \mu) = (\Phi_n g\iit\Phi_n\ii, \mu).$$
Moreover, recall that there is a map $d:\CU\to\GL_1$ given by
$$(g, \mu)\mapsto \mu^2,\qquad c\mapsto -1.$$
If $\pi$ is a cuspidal automorphic representation of the form considered in \Cref{Gal-reps}, and if $\pi$ is $C$-algebraic (often satisfied by Remark \ref{mult2}), then its associated Galois representation should be valued in $\CU = \CU(a,b)$.  In this subsection, we prove the following theorem, which is exactly Conjecture 5.3.4 of \cite{BuG}:

\begin{theorem}\label{c-valued-cohomological}
	Let $\pi$ be a $C$-algebraic cuspidal automorphic representation of $\U(a,b)(\A_{F^+})$ such that, for each archimedean place $v$, $\pi_v$ lies in the discrete series. Then, for each prime $\l$ at which $\pi$ is unramified, there exists a continuous Galois representation
	$$R_{\pi}:\GalF\to \CU(\Qlb)$$
	such that:
	\begin{enumerate}
		\item The composition of  $R_{\pi}$ with the projection $\CU(\Qlb) \to \GalF$ is the identity.
		\item The composition of $R_{\pi}$ with the map $d:\CU\to\GG_m$ is the cyclotomic character $\epsilon$.
		\item $R_{\pi}$ satisfies local-global compatibility at unramified primes: for each place $v$ of $F^+$ lying over a rational prime $p\ne \l$ at which both $F$ and $\pi$ are unramified, the local representation
$(R_{\pi}|_{W_{F^+_v}})^{\mathrm{ss}}$ is $\widehat{\widetilde{\U(a,b)}}$-conjugate to the representation sending $w\in W_{F^+_v}$ to $r_{\pi_v}(w) \hat{\xi}(|w|^{1/2})$, where $r_{\pi_v}$ is the local Langlands correspondence normalised as in section 2.2 of \cite{BuG} and $\hat{\xi}$ is the map $\C\t \to \widehat{\widetilde{\U(a,b)}}$, defined in \Cref{twisting}.
		\item For any complex conjugation $c$ the image $R(c)$ is $\widehat{{\U(a,b)}}(\Qlb)$-conjugate to $(\Phi_n^{-1},(-1)^{1/2}) c$.
	\end{enumerate}
\end{theorem}

The theorem follows from the following proposition, which is essentially a combination of \cite[Lemma 2.1.1]{CHT} and \cite[Section 8.3]{BuG}.

\begin{proposition}\label{polarisable-sign-bijection}
Let $k$ be a ring that is closed under taking square roots. Let $\chi:\GalF\to\GL_1(k)$ be a character. There is a bijection between:
\begin{enumerate}
	\item Isomorphism classes of representations 
	$$R: \GalF\to \CU(k)$$
	taken up to $\widehat{G}$ conjugacy, such that the composite of $R$ and the projection onto $\GalF$ is the identity and such that $d\circ R= \chi$.
	\item Isomorphism classes of polarised Galois representations $(\rho, \chi^{n-1}, \langle\cdot,\cdot\rangle)$.
\end{enumerate}
Moreover, $(\rho, \chi^{1-n}, \langle\cdot,\cdot\rangle)$ has sign $-\mu_c^2$ if and only if $R(c)$ has a representative of the form $(A\Phi_n\ii, \mu_c)c$, where $A \in \GL_n(k)$ defines the pairing $\langle\cdot,\cdot\rangle$.
\end{proposition}

\begin{remark}
	The fact that $(\rho, \chi^{n-1}, \langle\cdot,\cdot\rangle)$ is polarised is crucial to this proposition. For example, an essentially conjugate self dual representation $\rho$ such that $\rho\simeq\rho_1\+\rho_2$, where $\rho_1$ is even and $\rho_2$ is odd, would not lift to a representation valued in $\CU(k)$.
\end{remark}

\begin{proof}
		There is an isomorphism
		\begin{align*}
			\theta:\frac{\GL_n\times\GL_1}{\langle (-I_n)^{n-1}, -1\rangle}&\to\GL_n\times\GL_1\\
			(g, \mu)&\mapsto(g\mu^{1-n}, \mu^2)
		\end{align*}
		Let $\pr$ be the composition of $\theta$ with the projection to $\GL_n$. Let $R:\GalF\to\CU(k)$ be as in the proposition.	For $\delta\in \Ga F$, define:
		\begin{enumerate}
			\item $\rho(\delta) = \pr(R(\delta))$;
			\item $\langle x, y\rangle = x^t A\ii y$, where $R(c) = (A\Phi_n\ii, \mu_c)c$.
		\end{enumerate}
	
		Note that, since $R(c)^2 = (AA\iit(-I_n)^{n-1}, \mu_c^2)= (I_n, 1)$, we have $AA\iit = -\mu_c^2$, so that $(\rho, \chi^{1-n}, \langle\cdot,\cdot\rangle)$ has sign $-\mu_c^2$.
	

		Conversely, given a polarised representation $(\rho, \chi^{1-n}, \langle\cdot,\cdot\rangle)$, for $\delta\in \Ga F$, define
		$$R(\delta) = (\rho(\delta)\chi(\delta)^{(n-1)/2}, \chi(\delta)^{1/2})\delta$$
		for $\delta\in\Ga F$ and
		$$R(c) = (A\chi(c)^{(n-1)/2}\Phi_n\ii, \chi(c)^{1/2})c.$$
		
		Note that, by definition, $(\rho, \chi^{1-n}, \langle\cdot,\cdot\rangle)$ has sign $-\chi(c) = -(\chi(c)^{1/2})^2$, as required.
		
\end{proof}

We deduce \Cref{c-valued-cohomological}:

\begin{proof}[Proof of \Cref{c-valued-cohomological}]
	By \Cref{Gal-reps} and \cite{BCh}, there is a totally odd polarised Galois representation $(\rho_{\pi}, \epsilon^{1-n}, \langle\cdot,\cdot\rangle)$ attached to $\pi$.	By \Cref{polarisable-sign-bijection}, this representation lifts uniquely to a representation
	$$R_{\pi}:\GalF \to \CU(\Qlb)$$
	such that $d\circ R_{\pi} = \epsilon$, such that the projection to $\GalF$ is the identity and such that $R(c)=(A \Phi_n^{-1}, (-1)^{1/2})c$ with $A$ symmetric. 	All non-singular symmetric matrices are congruent over $\Qlb$ and, if $I_n = hAh^t$ for a matrix $h$, then 
	$$(\Phi_n\ii, (-1)^{1/2}) = (h, 1) (A \Phi_n^{-1}, (-1)^{1/2})c(h\ii, 1)c.$$
	Hence, $R$ satisfies conditions $(1), (2)$ and $(4)$.
%
It remains to check local-global compatibility at unramified places. Note that the bijection of \Cref{polarisable-sign-bijection} is induced from a sequence of $L$-homomorphisms $\CU(a,b)\to\LU(a,b)\to {}^L(\mathrm{Res}_{F/F^+}(\GL_n))$. For places $v$ inert in $F/F^+$ $(\rho_{\pi}\otimes \epsilon^{(n-1)/2})|_{G_{F_v^+}}$ is conjugate self-dual of parity 1 in the sense of \cite{Mok2015} (2.2.4) and (2.2.5), as the sign of $(\rho_\pi, \epsilon^{1-n},\langle\cdot,\cdot\rangle)$ is $-\epsilon(c)=+1$. Therefore, for both inert and split primes, it suffices  to compare the Langlands parameters under the injection $\Phi(\U(a,b)) \hookrightarrow \Phi(\mathrm{Res}_{F/F^+}(\GL_n))=\Phi({\GL_n}_{/F})$ arising from base change (as opposed to twisted base change; see \cite{Mok2015} Lemma 2.2.1).  By the proof of \Cref{polarisable-sign-bijection}, $R(\delta) = (\rho_{\pi, \ell}(\delta)\epsilon(\delta)^{(n-1)/2}, \epsilon(\delta)^{1/2})\delta$
		for $\delta\in\Ga F$. Hence, this compatibility follows from that of Theorem \ref{Gal-reps}.
\end{proof}

\begin{remark}
	The arguments of this section break down if we try to work with automorphic representations of $\GU(a,b)$ instead of automorphic representations of $\U(a,b)$. On the automorphic side, the base change map from automorphic representations of $\GU(a,b)$ to automorphic representations of $\GL_n\times\GL_1$ is two-to-one (whereas the base change map from $\U(a,b)$ to $\GL_n$ is injective). In particular, there should be two distinct lifts of the $\GL_n$ valued representation $\rho_{\pi}$ to a $\CGU$-valued representation. An analogous version of \Cref{polarisable-sign-bijection} indeed gives a two-to-one map from suitable $\CGU$-valued representations to suitable polarised Galois representations which preserves the sign at infinity. However, we are unable to show that either of the two lifts of $\rho_{\pi}$ match up to $\pi$ at all unramified primes. In \Cref{invariant-theory-remark}, we will see that the invariant theory of $\GU(a,b)$ suggests that this question is genuinely difficult.
\end{remark}

\subsection{Galois representations attached to polarised automorphic representations of $\GL_n$} \label{s2.5}

Let $F$ be a CM field and let $\pi$ be a cuspidal automorphic representation of $\GL_n(\AF)$. Assume that $\pi$ is (essentially) conjugate self-dual, i.e. that there exists a Hecke character $\chi: \A_F\t \to \C\t$ such that $\pi^c=\pi^{\vee} \otimes \chi$. Furthermore assume that there exists a character $\chi_0: \A_{F^+}\t/(F^+)\t \to \C\t$ such that $\chi=\chi_0 \circ \mathrm{Nm}_{F/F^+}\circ \mathrm{det}$ and $\chi_{0,v}(-1)$ is independent of $v \mid \infty$.

Following \cite{FP}, we say that $\pi$ is weakly regular if all its infinitesimal characters for $F_v$ with $v \mid \infty$ are of the form $a_v=(a_{1,v}, \ldots, a_{n,v})$, where the $a_{i,v}$ have multiplicity at most two. Moreover, we say that $\pi$ is odd if the Asai $L$-function $L(s, \pi, \mathrm{Asai}^{(-1)^{n-1}\epsilon(\chi_0)} \otimes \chi_0^{-1})$ has a pole at $s=1$. For precise definitions of this Langlands $L$-function, the representations $\mathrm{Asai}^{\pm}: {}^L\mathrm{Res}_{F/F^+}(\GL_n) \to \GL(\C^n \otimes \C^n)$ and the sign $\epsilon(\chi_0)$, we refer to \cite[Section 9.1]{FP}. We note here that the Rankin--Selberg $L$-function $$L(\pi \otimes \pi^{\vee},s)=L(s, \pi, \mathrm{Asai}^+ \otimes \chi_0^{-1}) L(s, \pi, \mathrm{Asai}^- \otimes \chi_0^{-1})$$ has a simple pole since $\pi$ is cuspidal and neither of the Asai $L$-values vanishes at $s=1$.

\cite{FP} combine the results of \cite{GoldringKoskivirta} and \cite{Pilloni-Stroh} (see Theorem \ref{GKPS}) with Mok's proof in \cite{Mok2015} of the Arthur classification for quasi-split unitary groups to obtain:

\begin{theorem}[\cite{FP} Theorem 9.10]
Let $F$ be a CM field and $\pi$ be a weakly regular $C$-algebraic odd cuspidal automorphic representation of $\GL_n(\AF)$, such that $\pi^c=\pi^{\vee} \otimes \chi$ for $\chi: \A_F\t \to \C\t$ as above. Then there exists a Galois representation
	$$\rho_{\pi} :\Ga F\to \GL_n(\Qlb)$$
	that is unramified at all finite places $w\nmid \l$ at which $\pi$ is unramified, satisfies local-global compatibility up to semisimplification--i.e. $(\rho_{\pi}|_{W_{F_w}})^{\mathrm{ss}} \cong \rec_{F_w}(\pi_w \otimes | \cdot |_w^{\frac{1-n}{2}})$-- at unramified places $w$, and is such that $\rho_{\pi}^c \simeq \rho_{\pi}\dual\tensor \rho_{\chi} \epsilon^{1-n}.$
\end{theorem}

\cite{FP} Theorem 9.11 also proves a result towards local-global compatibility at places dividing $\ell$.

\begin{corollary}
	\Cref{main-thm-intro} implies \Cref{BC-signs-intro}.
\end{corollary}

\begin{proof}
As explained in \cite[Theorem 9.6]{FP}  the results of \cite{Mok2015} imply that a $C$-algebraic odd representation $\pi$ of $\GL_n(\AF)$ descends to a $C$-algebraic representation $\tilde \pi$ of the quasi-split unitary group $\U(n)/F^+$ (which equals $\U(n/2,n/2)$ for $n$ even and $\U(\frac{n+1}{2}, \frac{n-1}{2})$ for $n$ odd). The definition of weakly regular is chosen exactly to ensure that this descent is a non-degenerate limit of discrete series (see  \Cref{mult2}).
\end{proof}
\section{Lafforgue pseudocharacters and invariant theory}\label{invariants}

In this section, we prove that a limit of odd representations is odd, so that Galois representations attached to irregular automorphic representations are also odd. Our method is to reconstruct the Galois representations by using Lafforgue pseudocharacters in place of Taylor's pseudocharacters. This method was previously applied in \cite{weiss2018image} to prove that the Galois representations attached to Siegel modular forms are valued in $\Gf$.

Lafforgue pseudocharacters were introduced by Vincent Lafforgue as part of his proof of the automorphic-to-Galois direction of the geometric Langlands correspondence for a general reductive group \cite{lafforgue2012chtoucas}. Rather than following Lafforgue's original approach \cite[Section 11]{lafforgue2012chtoucas}, we use a categorical approach due to Weidner \cite{Weidner}. Our exposition follows that of \cite{weissthesis}.

\subsection{$\FFS$-algebras}

Let $\FFS$ be the category of free, finitely-generated semigroups. If $I$ is a finite set, let $\FS(I)$ denote the free semigroup generated by $I$.

If $I\to J$ is a morphism of sets, then there is a corresponding semigroup homomorphism $\FS(I)\to\FS(J)$. However, not all morphisms in $\FFS$ are of this form.

\begin{lemma}[{\cite[Lemma 2.1]{Weidner}}]\label{FFS-generators}
	Any morphism in $\FFS$ is a composition of morphisms of the following types:
	\begin{itemize}
		\item morphisms $\FS(I) \to\FS(J)$ that send generators to generators, i.e. those induced by morphisms $I\to J$ of finite sets;
		\item morphisms
		\begin{align*}
		\FS(\set{x_1,\ldots,x_n})&\to \FS(\set{y_1,\ldots,y_{n+1}})\\
		x_i&\mapsto\begin{cases}
		y_i&i<n\\
		y_ny_{n+1}&i=n.
		\end{cases}
		\end{align*}
	\end{itemize}
\end{lemma}

\begin{definition}
	Let $R$ be a ring. An \emph{$\FFS$-algebra} is a functor from $\FFS$ to the category $\Ralg$ of $R$-algebras. A morphism of $\FFS$-algebras is a natural transformation of functors.
\end{definition}

We will be interested in the following two examples:

\begin{examples}
	\mbox{}
	\begin{enumerate}
		\item Let $\Gamma$ be a group and let $A$ be an $R$-algebra. We define a functor
		$$\Map(\Gamma\tdo, A):\FFS\to\Ralg$$
		as follows.	For each finite set $I$, let $\Map(\Gamma^I, A)$ denote the $R$-algebra of set maps $\Gamma^I\to A$. The functor
		$$\Map(\Gamma\tdo,A):\FS(I)\mapsto \Map(\Gamma^I, A)$$
		is an $\FFS$-algebra.

		If $\Gamma$ is a topological group and $A$ is a topological $R$-algebra, then this construction works with continuous maps in place of set maps.
		
		\item Let $G, X$ be affine group schemes over $R$, and let $G$ act on $X$. For any finite set $I$, $G$ acts diagonally on $X^I$, and, hence, $G$ acts on the coordinate ring $R[X^I]$ of $X^I$.
		
		For each finite set $I$, let $R[X^I]^G$ be the $R$-algebra of fixed points of $R[X^I]$ under the action of $G$. A morphism $\phi:\FS(I)\to\FS(J)$  in $\FFS$ induces a morphism of $R$-schemes $X^J\to X^I$, and thus a $R$-algebra morphism $R[X^I]^G\to R[X^J]^G$. The corresponding functor
		$$R[X\tdo]^G:\FS(I)\mapsto R[X^I]^G$$
		is an $\FFS$-algebra.
	\end{enumerate}
\end{examples}

\subsection{Lafforgue pseudocharacters}

Let $R$ be a ring and let $G$ be a reductive group over $R$. Let $G^\circ$ denote the identity connected component of $G$, which we assume is split. Then $G^\circ$ acts on $G$ by conjugation, and we can form the $\FFS$-algebra $R[G\tdo]^{G^\circ}$.

\begin{definition}
	Let $\Gamma$ be a group and let $A$ be a $R$-algebra. A \emph{$G$-pseudocharacter} of $\Gamma$ over $A$ is an $\FFS$-algebra morphism
	$$\Theta\tdo:R[G\tdo]^{G^\circ}\to\Map(\Gamma\tdo,A).$$
\end{definition}

\begin{remarks}
	\mbox{}
	\begin{enumerate}
		\item Unwinding this definition recovers Lafforgue's original definition \cite[Definition 11.3]{lafforgue2012chtoucas}. Indeed, Lafforgue defines a pseudocharacter as a collection $(\Theta_n)_{n\ge 1}$ of algebra maps
		$$\Theta_n:R[G^n]^{G^\circ}\to \Map(\Gamma^n, A)$$
		that are compatible in the following sense:
		\begin{enumerate}
			\item If $n, m\ge1$ are integers and $\zeta:\{1,\ldots,m\}\to\{1, \ldots,n \}$, then for every $f\in R[G^m]^{G^\circ}$ and $\gamma_1,\ldots\gamma_n\in \Gamma$, we have
			$$\Theta_n(f^\zeta)(\gamma_1,\ldots,\gamma_n) = \Theta_m(f)(\gamma_{\zeta(1)},\ldots, \gamma_{\zeta(m)}),$$
			where $f^\zeta(g_1,\ldots, g_n) = f(g_{\zeta(1)},\ldots,g_{\zeta(m)})$.
			\item For every integer $n\ge 1$, $f\in R[G^n]^{G^\circ}$ and $\gamma_1,\ldots\gamma_{n+1}\in \Gamma$, we have
			$$\Theta_{n+1}(\hat{f})(\gamma_1,\ldots,\gamma_{n+1}) = \Theta_n(f)(\gamma_1, \ldots, \gamma_{n-1}, \gamma_n\gamma_{n+1}),$$
			where $\hat{f}(g_1, \ldots, g_{n+1}) =f(g_1, \ldots, g_{n-1}, g_ng_{n+1})$. 
		\end{enumerate}
		
		By definition, an $\FFS$-algebra morphism $R[G\tdo]^{G^\circ}\to\Map(\Gamma\tdo,A)$ consists of a collection of $R$-algebra morphisms $\Theta^I:R[G^I]^{G^\circ}\to \Map(\Gamma^I, A)$, such that for any semigroup homomorphism $\phi:\FS(I)\to\FS(J)$, the following diagram commutes:
		\begin{center}
			\begin{tikzcd}
			R[G^I]^{G^\circ}\arrow[r, "\Theta^I"]\arrow[d]&\Map(\Gamma^I, A)\arrow[d]\\
			R[G^J]^{G^\circ}\arrow[r, "\Theta^J"]&\Map(\Gamma^J, A)
			\end{tikzcd}
		\end{center}
		Here, the vertical arrows are those induced by $\phi$. By \Cref{FFS-generators}, checking that this diagram commutes for all morphisms $\phi$ is equivalent to verifying conditions (a) and (b) above.
		
		\item Suppose that $G$ is a connected linear algebraic group with a fixed embedding ${G\hookrightarrow\GL_r}$ for some $r$. Let $\chi$ denote the composition of this embedding with the usual trace function. Then $\chi\in\Z[G]^G$, and if $\Theta\tdo$ is a $G$-pseudocharacter of $\Gamma$ over $A$, then  
		$$\Theta^{1}(\chi)\in \Map(\Gamma, A)$$
		is a classical pseudocharacter. In fact, we will see in \Cref{CG-pseudo} that when $G=\GL_n$, $\Theta\tdo$ is completely determined by $\Theta^{1}(\chi)$ \cite[Remark 11.8]{lafforgue2012chtoucas}. In particular, the notion of a $G$-pseudocharacter is a generalisation of the notion of a classical pseudocharacter. 
		
	\end{enumerate}
\end{remarks}

We finish this subsection by recording a generalisation of \cite[Lemma 1]{taylor1991galois}, which notes that we are free to change the $R$-algebra $A$. Part (i) is part of \cite[Lemma 4.4]{bockle2016}.

\begin{lemma}\label{lafforgue-factor}
	Let $A$ be an $R$-algebra, and let $\Gamma$ be a group. 
	\begin{enumerate}
		\item Let $h:A\to A'$ be a morphism of $R$-algebras, and let $\Theta\tdo$ be a $G$-pseudocharacter of $\Gamma$ over $A$. Then $h_*(\Theta) = h\circ\Theta\tdo$ is a $G$-pseudocharacter of $\Gamma$ over $A'$.
		
		\item Let $h:A\hookrightarrow A'$ be an injective morphism of $R$-algebras. Define a collection of maps $\Theta\tdo$, where, for each finite set $I$, $\Theta^I:R[G^I]^{G^\circ}\to \Map(\Gamma^I, A)$ is a map of sets.
		
		Suppose that $h\circ\Theta\tdo$ is a $G$-pseudocharacter of $\Gamma$ over $A'$. Then $\Theta\tdo$ is a $G$-pseudocharacter over $A$.
	\end{enumerate}
\end{lemma}

%

\subsection{Lafforgue pseudocharacters and $G$-valued representations}

The key motivation for introducing Lafforgue pseudocharacters is their connection to $G$-valued representations. From now on, assume that $R=\Z$, so that $G$ is a reductive group over $\Z$ with $G^\circ$ split, and $A$ is a ring.

\begin{lemma}\label{Lafforgue-rep-to-pseudo}
	Let $\rho:\Gamma\to G(A)$ be a representation of $\Gamma$. Define
	$$(\Tr\rho)\tdo:\Z[G\tdo]^{G^\circ}\to \Map(\Gamma\tdo, A)$$
	by
	$$(\Tr\rho)^I(f)\big((\gamma_i)_{i\in I}\big)=f\big((\rho(\gamma_i))_{i\in I}\big)$$
	for each finite set $I$ and for each $f\in\Z[G^I]^{G^\circ}$.
	
	Then $(\Tr\rho)\tdo$ is a $G$-pseudocharacter of $\Gamma$ over $A$. Moreover, $(\Tr\rho)\tdo$ depends only on the $G^\circ(A)$-conjugacy class of $\rho$.
\end{lemma}

In fact, in many cases, the converse of \Cref{Lafforgue-rep-to-pseudo} is also true. Let $k$ be an algebraically closed field and let $\rho: \Gamma\to G(k)$ be a representation of $\Gamma$. If $G = \GL_n$ and $\rho$ is semisimple, then Taylor \cite{taylor1991galois} proved that $\rho$ can be recovered from its classical pseudocharacter. To state the generalisation of this fact for $G$-pseudocharacters, we first define what it means for $\rho$ to be semisimple in general.

\begin{definition}[{\cite[Definitions 3.3, 3.5]{bockle2016}}]\label{def-g-irred}
	Let $H$ denote the Zariski closure of $\rho(\Gamma)$.
	\begin{enumerate}
		\item We say that $\rho$ is \emph{$G$-irreducible} if there is no proper parabolic subgroup of $G$ containing $H$.
		\item We say that $\rho$ is \emph{semisimple} or \emph{$G$-completely reducible} if, for any parabolic subgroup $P\sub G$ containing $H$, there exists a Levi subgroup of $P$ containing $H$.
	\end{enumerate}
\end{definition}

\begin{theorem}[{\cite[Proposition 11.7]{lafforgue2012chtoucas}}, {\cite[Theorem 4.5]{bockle2016}}]\label{lafforge-to-rep}
	Let $k$ be an algebraically closed field. The assignment $\rho\mapsto(\Tr\rho)\tdo$ defines a bijection between the following two sets:
	\begin{enumerate}
		\item The set of $G^\circ(k)$-conjugacy classes of $G$-completely reducible homomorphisms ${\rho:\Gamma\to G(k)}$;
		\item The set of $G$-pseudocharacters $\Theta\tdo:\Z[G\tdo]^{G^\circ}\to\Map(\Gamma\tdo,k)$ of $\Gamma$ over $k$.
	\end{enumerate}
\end{theorem}

\subsection{$\CU$-pseudocharacters}\label{CG-pseudo}

\begin{definition}[{\cite[Definition 2.6]{Weidner}}]
	Let $A\tdo$ be an $\FFS$-algebra. Given a subset $\Sigma\sub \bigcup_{I}A^I$, define the \emph{span} of $\Sigma$ in $A\tdo$ to be the smallest $\FFS$-subalgebra $B\tdo$ of $A\tdo$, such that $\Sigma\sub \bigcup_IB^I$. We say that $\Sigma$ generates $A\tdo$ if the span of $\Sigma$ in $A\tdo$ is the whole of $A\tdo$.
\end{definition}

\begin{example}
	Suppose that $k$ is a field of characteristic $0$. By results of Procesi \cite{Procesi-invariants}, the $\FFS$-algebra $k[\GL_n\tdo]^{\GL_n}$ is generated by the elements $\Tr,\det\ii\in k[\GL_n]^{\GL_n}$. If $\FFG$ denotes the category of free, finitely-generated groups, then we can define $\FFG$-algebras analogously to $\FFS$-algebras and, as an $\FFG$-algebra, $k[\GL_n\tdo]^{\GL_n}$ is generated by $\Tr$.\footnote{If $X\in \GL_n$, then $\det(X)$ can be expressed as a polynomial in $\Tr(X^i)$. Hence, as an $\FFG$-algebra, $\det\ii$ is in the $\FFG$-subalgebra generated by $\Tr$.} Hence, if $\Theta\tdo$ is any $\GL_n$-pseudocharacter, by \cite[Corollary 2.13]{Weidner}, $\Theta\tdo$ is completely determined by its classical pseudocharacter $\Theta^{1}(\Tr)\in\Map(\Gamma, k)$.
	
	More generally, if $R$ is a ring, then $R[\GL_n\tdo]^{\GL_n}$ is generated by $s_i, \det\ii\in R[\GL_n]^{\GL_n}$, where $s_i$ is the $i\th$ coefficient of the characteristic polynomial \cite{procesiconcini} .
\end{example}

\begin{lemma}\label{invariants}
	As an $\FFS$-algebra, $\Z[\CU\tdo]^{\widehat{\widetilde{U}}}$ is generated by the elements
	\begin{alignat*}{2}
		&(g, \mu)\;\mapsto \;s_i(g\mu^{1-n}) &&\quad (g, \mu)c\;\mapsto\; 0,\\
		&(g, \mu)\;\mapsto\; \det(g\mu^{1-n})\ii &&\quad (g, \mu)c\;\mapsto\; 0,\\
		&(g, \mu)\;\mapsto\; \mu^{\pm2}&&\quad(g, \mu)c\;\mapsto \;0,\\
		&(g,  \mu)\;\mapsto\; 0&&\quad(g, \mu)c\;\mapsto \;\mu^{\pm2},
	\end{alignat*} 
	of $\Z[\CU]^{\widehat{\widetilde{U}}}$.
\end{lemma}

\begin{proof}
	Recall that we have an isomorphism
	$$\frac{\GL_n\times\GL_1}{\langle(-I_n)^{n-1}, -1\rangle}\to\GL_n\times\GL_1$$
	given by
	$$(g, \mu)\mapsto (g\mu^{1-n}, \mu^2).$$
	Transporting the action of $\GalF$ on $\frac{\GL_n\times\GL_1}{\langle (-I_n)^{n-1}, -1\rangle}$ to $\GL_n\times\GL_1$, we see that 
	$$\CU\cong (\GL_n\times\GL_1)\rtimes\GalF$$
	where $\GalF$ acts on $(g, \mu)\in \GL_n\times\GL_1$ via its quotient $\Gal(F/F^+)$, and $c\in\Gal(F/F^+)$ acts as
	$$c\cdot(g, \mu) = (g\iit \mu^{n-1}, \mu) .$$
%
	For an affine scheme $X_{/\Z}$, write $\Z[X]$ for the $\Z$-algebra such that $X = \Spec(\Z[X])$. Let $H = {\GL_n\times\GL_1}$. Fix $r\iZ$. Let 
	$$\CU^r\sslash H := \Spec(\Z[\CU^r]^{H}),$$
	where $H$ acts by pointwise conjugation, and let
	$$\pi: \CU^r\to \CU^r\sslash H$$
	be the quotient map. As a $\Z$-scheme, 
	$$\CU^r = \bigsqcup_{x\in \Gal(F/F^+)^r} (H^r)x ,$$
	where the subsets $(H^r)x \sub \CU^r$ are closed and pairwise disjoint. Hence,
	$$\Z[\CU^r] \cong \prod_{x\in \Gal(F/F^+)^r}\Z[(H^r)x].$$
	Moreover, the subsets $(H^r)x$ are stable under the conjugation action of $H$. Hence, by \cite[Theorem 3]{Seshadri}, the subsets $\pi((Hx)^r)$ are closed, disjoint subsets of $\CU^r\sslash H$ and, since $\pi$ is surjective, we see that  
	$$\CU^r\sslash H = \bigsqcup_{x\in \Gal(F/F^+)^r}\pi((H^r)x).$$
	It follows that
	$$\Z[\CU^r]^{H}\cong \prod_{x\in \Gal(F/F^+)^r}\Z[(H^r)x]^H.$$
	
	Consider a component $(H^r)x$, where $x = (\underbrace{1,\ldots, 1}_{r_1\text{ times}}, \underbrace{c, \ldots, c}_{r_2\text{ times}})$. Recall that $(\gamma, \nu)\in\GL_n\times\GL_1$ acts by conjugation on $H$ and as
	$$(g, \mu)c\mapsto (\gamma g\gamma^t\nu^{n-1}, \mu) = ((\gamma\nu^{(n-1)/2})g (\gamma\nu^{(n-1)/2})^t, \mu)$$
	on $Hc$. Since $\GL_n\times\GL_1$ acts trivially on the $\GL_1$ component, we have
	$$\Z[Hc]^H = \Z[\GL_1]\tensor\Z[\GL_n]^{H},$$
	where the action of $H$ on $\GL_n$ is given by $(\gamma, \nu)\cdot g = (\gamma\nu^{(n-1)/2})g(\gamma\nu^{(n-1)/2})^t$. In particular, this action of $H$ on $\GL_n$ factors through an action of $\GL_n$ on $\GL_n$ via the map that sends $(\gamma, \nu)\in\GL_n\times\GL_1$ to $\gamma\nu^{(n-1)/2}$ for any choice of square root of $\nu$.
	
	We see that
	\[\Z[(H^r)x]^H = \Z[\GL_1^r]\tensor\Z[\GL_n^{r_1}\times\GL_n^{r_2}]^{\GL_n},\]
	where the action of $\GL_n$ on the first $r_1$ $\GL_n$'s is by conjugation and the action on the second $r_2$ $\GL_n$'s is by $\gamma\cdot g = \gamma g \gamma^t$.
	
	Now, the map $\GL_n\hookrightarrow\M_n\times\M_n$ that sends $g\mapsto (g, g\ii)$ is a closed embedding. Hence, by \cite[Theorem 3]{Seshadri}, the restriction map
	\[\Z[\M_n^{r_1}\times\M_n^{r_1}\times\M_n^{r_2}\times\M_n^{r_2}]^{\GL_n}\to \Z[\GL_n^{r_1}\times\GL_n^{r_2}]^{\GL_n},\]
	is a surjection.
	
	Consider $\Z[\M_n^{r_1}\times\M_n^{r_1}\times\M_n^{r_2}\times\M_n^{r_2}]^{\GL_n}$. Here, $\gamma \in \GL_n$ acts on $(g_1,\ldots, g_{2r_1}, h_1,\ldots, h_{r_2}, h'_1,\ldots, h'_{r_2})$ by taking it to
	\[ (\gamma g_1\gamma\ii,\ldots, \gamma g_{2r_1}\gamma\ii, \gamma h_1\gamma^t,\ldots, \gamma h_{r_2}\gamma^t, \gamma\iit h'_1\gamma\ii,\ldots, \gamma\iit h'_{r_2}\gamma\ii)\]
	Let $V=\Z^n$. We identify $\M_n(\Z) = \End(V)$. We can identify:
	\begin{itemize}
		\item $V\tensor V\dual \cong \End(V)$ via $(v\tensor\phi)\mapsto (F_{(v,\phi)}:u\mapsto \phi(u)v)$. Under this isomorphism, $(\gamma v\tensor \gamma\phi)\mapsto\gamma F_{(v,\phi)}\gamma\ii$
		\item Identifying $V\cong V\dual$ using the standard inner product, we have a $\GL_n$-isomorphism $V\tensor V\cong \End(V)$, where $\gamma\in \GL_n$ acts on $F\in\End(V)$ by $\gamma\cdot F(v) = \gamma F(\gamma^tv)$.
		\item Identifying $V\cong V\dual$ using the standard inner product, we have a $\GL_n$-isomorphism $V\dual\tensor V\dual\cong \End(V)$, where $\gamma\in \GL_n$ acts on $F\in\End(V)$ by $\gamma\cdot F(v) = \gamma\iit F(\gamma\ii v)$.
	\end{itemize}
	
	Hence, $\Z[\M_n^{r_1}\times\M_n^{r_1}\times\M_n^{r_2}\times\M_n^{r_2}]^{\GL_n}\cong\Z[(V\tensor V\dual))^{2r_1}\times (V\tensor V)^{r_2}\times (V\dual\tensor V\dual)^{r_2}]^{\GL_n}$. Applying \cite[Theorem 3.1]{procesiconcini}, we see that this ring is generated by elements of the form $\phi(v)$,	where $\phi\in V\dual$, $v\in V$.
	Unwinding, we find that the invariants of $(A_1\,\ldots A_{2r_1}, B_1,\ldots B_{r_2}, C_1,\ldots C_{r_2})\in \M_n^{r_1}\times\M_n^{r_1}\times\M_n^{r_2}\times\M_n^{r_2}$ are polynomials in $s_i(M)$ where $M$ is in the free semigroup generated by $\{A_i, B_iC_j, B_iC_j^t, B_i^tC_j\}$. 
	
	Unwinding further, we see that $\Z[H^rx]^{H}$ is generated by maps which send 
$$((g_1,\mu_1),\ldots (g_{r_1}, \mu_{r_1}),(h_1,\mu_1)c,\ldots (h_{r_2}, \mu_{r_2})c)$$ to:
	\begin{enumerate}
		\item $s_i(M)$, $\det\ii(M)$, where $s_i$ is the $i\th$ coefficient of the characteristic polynomial and $M$ is in the free group generated by $\{g_i, h_ih_j{}\ii, h_ih_j{}\iit\}$;
		\item elements of the free group generated by the $\mu_i$.
	\end{enumerate}
	
	Finally, observing that
	\[(h_i, \mu_i)c \cdot (h_j, \mu_j)c = (h_ih_j\iit\mu_j^{1-n}, \mu_i\mu_j)\]
	\[(h_i, \mu_i)c \cdot((h_j, \mu_j)c)\ii = (h_ih_j\ii, \mu_i\mu_j{\ii})\]
	and unwinding a final time, we deduce that, as an $\FFS$-algebra, $\Z[\CU\tdo]^{\widehat{\widetilde{U}}}$ is generated by the elements
	\begin{alignat*}{2}
		&(g, \mu)\;\mapsto \;s_i(g\mu^{1-n}) &&\quad (g, \mu)c\;\mapsto\; 0,\\
		&(g, \mu)\;\mapsto\; \det(g\mu^{1-n})\ii &&\quad (g, \mu)c\;\mapsto\; 0,\\
		&(g, \mu)\;\mapsto\; \mu^{\pm2}&&\quad(g, \mu)c\;\mapsto \;0,\\
		&(g,  \mu)\;\mapsto\; 0&&\quad(g, \mu)c\;\mapsto \;\mu^{\pm2},
	\end{alignat*} 
	as required.
\end{proof}

\subsection{Oddness in low weight}

We can now prove \Cref{main-thm-intro}.

\begin{proof}[Proof of \Cref{main-thm-intro}]
	The result follows by replacing Taylor's pseudocharacters with Lafforgue's pseudocharacters in the proof of \cite[Theorem 10.5.3]{GoldringKoskivirta}. Let $\T$ be the abstract Hecke algebra generated by the Hecke operators of $\pi$, let $E$ be the finite extension of $\Ql$ generated by the Hecke parameters of $\pi$, and let $\theta:\T\to\O_E$ be the Hecke map associated to $\pi$. Then \cite{GoldringKoskivirta} consider the reductions modulo $\l^m$, which correspond to torsion eigenclasses in coherent cohomology of the reduction of the Shimura variety corresponding to $U(a,b)$ modulo $\l^m$.

In Theorem 10.4.1 they associate Galois representations to such torsion classes by proving that these Hecke maps factor through a Hecke algebra $\T_m$ acting on cuspidal automorphic representations with regular discrete series. In particular, they produce (see \cite[(10.6.2)]{GoldringKoskivirta}) a sequence of Galois representations
	$$\rho_m:\Ga F \to \GL_n(\T_m\tensor\Qlb)$$
	such that:
	\begin{itemize}
		\item $\T_m$ is the Hecke algebra (which \cite{GoldringKoskivirta} denote by $\mathcal{H}^{0,+}(\nu+ak\eta_\omega)$, with $a, k$ depending on $m$) parametrising automorphic representations of $\U(a,b)$ of a certain regular weight depending on $m$.
		\item For each $m$, the map $\T\to\O_E\to\O_E/\l^m$ factors through a map $r_m:\T_m\to\O_E/\l^m$. In other words, the eigenvalue system of $\pi$ is congruent modulo $\l^m$ to the eigenvalue system of a regular form $\pi_m$.
	\end{itemize}
By Theorem \ref{c-valued-cohomological} the Galois representation $\rho_m$ lifts to a representation 
$$R_m:\Ga F \to \CU(\T_m \otimes \Qlb)$$ with the correct sign.	Let $\Theta_m\tdo$ be the $\CU$-pseudocharacter of $\GalF$ over $\T_m\tensor\Qlb$ attached to $R_m$ by Theorem \ref{lafforge-to-rep}. By \Cref{invariants}, $\Theta_m\tdo$ is completely determined by
	$$\Theta_m(f):\GalF\to\T_m\tensor\Qlb,$$
	where $f$ is one of the elements
		\begin{alignat*}{2}
		&(g, \mu)\;\mapsto \;s_i(g\mu^{1-n}) &&\quad (g, \mu)c\;\mapsto\; 0,\\
		&(g, \mu)\;\mapsto\; \det(g\mu^{1-n})\ii &&\quad (g, \mu)c\;\mapsto\; 0,\\
		&(g, \mu)\;\mapsto\; \mu^{\pm2}&&\quad(g, \mu)c\;\mapsto \;0,\\
		&(g,  \mu)\;\mapsto\; 0&&\quad(g, \mu)c\;\mapsto \;\mu^{\pm2},
	\end{alignat*} 
	of $\Z[\CU]^{\widehat{\widetilde{U}}}$. Since the characteristic polynomial of $\rho_m$ has coefficients in $\T_m\tensor\Zlb$ and since $d \circ R_m=\epsilon$, it follows that, for each such $f$, $\Theta_m(f)$ factors through $\T_m\tensor\Zlb$. Thus, by \Cref{lafforgue-factor}, $\Theta_m\tdo$ is a $\CU$-pseudocharacter of $\GalF$ over $\T_m\tensor\Zlb$. In particular, this argument promotes the Taylor pseudocharacter of \cite[(10.6.3)]{GoldringKoskivirta} to a Lafforgue pseudocharacter and thereby strengthening \cite[Theorem 10.4.1]{GoldringKoskivirta}. Hence, composing $\Theta_m\tdo$ with the map $r_m$, we obtain a $\CU$-pseudorepresentation of $\GalF$ over $\O_E/\l^m$. Moreover, if $m'>m$, then 
	$$(r_m\circ\Theta_m)\tdo=(r_{m'}\circ\Theta_{m'})\tdo\pmod{\l^m}.$$
	Hence, we can form a $\CU$-pseudocharacter
	$$\Theta\tdo = \varprojlim_m (r_m\circ\Theta_m)\tdo$$
	of $\GalF$ over $\O_E$. Viewing $\O_E$ as a subalgebra of $\Qlb$ and applying \Cref{lafforge-to-rep}, we obtain the Galois representation
	$$R_{\pi}:\GalF\to\CU(\Qlb).$$
The fact that $R_{\pi}$ has the correct sign at infinity follows from the fact that $\Theta(f)$ is the limit of $\Theta_m(f)$ when $f$ is the map
	\begin{alignat*}{2}
	&(g,  \mu)\;\mapsto\; 0&&\quad(g, \mu)c\;\mapsto \;\mu^{\pm2}.
\end{alignat*} 
	The fact that $R_{\pi}$ satisfies local-global compatibility at unramified primes follows from the fact that $\Theta(f)$ is the limit of $\Theta_m(f)$, where $f$  is one of the elements
	\begin{alignat*}{2}
		&(g, \mu)\;\mapsto \;s_i(g\mu^{1-n}) &&\quad (g, \mu)c\;\mapsto\; 0,\\
		&(g, \mu)\;\mapsto\; \det(g\mu^{1-n})\ii &&\quad (g, \mu)c\;\mapsto\; 0,\\
		&(g, \mu)\;\mapsto\; \mu^{\pm2}&&\quad(g, \mu)c\;\mapsto \;0.
	\end{alignat*} 
\end{proof}

\subsection{$\GU(a,b)$-representations}\label{invariant-theory-remark}
	We conclude by highlighting a constraint to our approach and why it cannot be used to prove an analogous result for automorphic representations of $\GU(a,b)$.
	
	The key input to the above proof is that any element of $\Z[\CU\tdo]^{\widehat{\widetilde{U}}}$ is generated by elements of $\Z[\CU]^{\widehat{\widetilde{\U}}}$ and, therefore, that any $\CU$-pseudocharacter $\Theta\tdo$ is completely determined by its action on elements of $\Z[\CU]^{\widehat{\widetilde{U}}}$. Indeed, when $\Theta\tdo$ comes from a representation and $f$ is such an element, then the map $\Theta(f):\GalF\to\T\tensor\Qlb$ can be related to automorphic data by viewing its action on Frobenius elements. On the other hand, if $f\in \Z[\CU^r]^{\widehat{\widetilde{U}}}$ were an element that is not generated by elements of $\Z[\CU]^{\widehat{\widetilde{\U}}}$, then there is no obvious way to relate map $\Theta(f):\GalF^r\to\T\tensor\Qlb$ to automorphic data: on the automorphic side, it is unclear what the Hecke operators should be; and on the Galois side, it is unclear what the analogues of Frobenius elements should be for powers of the absolute Galois group.
	
	Hence, our arguments cannot be used to show that the Galois representation attached to an irregular automorphic representation of $\GU(a,b)$ has image in $\CGU(\Qlb)$, since, when $a+b$ is even, the $\FFS$ algebra $\Z[\CGU\tdo]^{\widehat{\widetilde{\GU}}}$ is \emph{not} generated by elements of $\Z[\CGU]^{\widehat{\widetilde{\GU}}}$. Indeed, similarly to the case of $\mathrm{SO}_{2n}$, (c.f. \cite[Lemma 18]{Weidner}), the full polarisation of the Pfaffian function $\mathrm{pf}(A-A^t)$ is an element of $\Z[\CGU^n]^{\widehat{\widetilde{\GU}}}$ that cannot be generated by elements of $\Z[\CGU]^{\widehat{\widetilde{\GU}}}$.
	
	Moreover, when $n = a+b$ is a multiple of $4$, the representations
	\[R_1, R_2: (\Z/4\Z)^2\to \CGU(\Qlb)\]
	defined by
	\[R_1: (0,1)\mapsto \br{\fourmat{I_m}{\Phi_m}\Phi_n, 1, 1}c; \quad (1,0)\mapsto\br{\fourmat{\Phi_m}{I_m}\Phi_n, 1, 1}c \]
	and
	\[R_2: (0,1)\mapsto \br{\fourmat{\zeta_nI_m}{\zeta_n\Phi_m}\Phi_n, 1, 1}c; \quad (1,0)\mapsto\br{\fourmat{\zeta_n\Phi_m}{\zeta_nI_m}\Phi_n, 1, 1}c,\]
	where $\zeta_n$ is a primitive $n\th$ root of unity, are everywhere locally conjugate\footnote{i.e. for every $a\in(\Z/4\Z)^2$, $R_1(a)$ and $R_2(a)$ are conjugate. See \cite{Larsen-conj}.}, but are not globally conjugate. In particular, if $\Theta\tdo$ is pseudocharacter attached to a $\GU$-valued representation, then the actions of $\Theta(f)$ on Frobenius elements for $f\in \Z[\CGU]^{\widehat{\widetilde{\GU}}}$ are not enough to uniquely determine $\Theta\tdo$.

\bibliography{bibliography}
\bibliographystyle{alpha}

\end{document}